\documentclass{article}

\usepackage[utf8]{inputenc}
\usepackage[english]{babel}
\usepackage{amsfonts}
\usepackage[11pt]{moresize}
\usepackage{pgfplots}
\usepackage{xfrac}
\newcommand{\Addresses}{{% additional braces for segregating \footnotesize
		\bigskip
		\footnotesize
		
		Humboldt-Universit\"at zu Berlin, Institut f\"ur Mathematik, Rudower Chausee 25
		\hfill \newline\texttt{}
		\indent 12489 Berlin, Germany} 
	\par\nopagebreak
	\textit{E-mail address}: \texttt{andreibud95@protonmail.com}
}

\usepackage{amssymb}
\usepackage{mathtools}
\usepackage{tikz}
\usepackage{tikz-cd}
\usepackage{mathtools}
\usepackage{amsmath}
\usepackage{amsthm}
\usepackage{hyperref}
\usepackage{float}
\usetikzlibrary{calc, positioning,decorations.markings,arrows}

\usepackage{graphics}
\usepackage{color}
\usepackage{tabularx,colortbl}
\usetikzlibrary{positioning}
\usetikzlibrary{decorations.markings,arrows}
\usetikzlibrary{calc}
\usepackage{amsthm}

\theoremstyle{plain}
\newtheorem{trm}{Theorem}[section]
\newtheorem{lm}[trm]{Lemma}

\theoremstyle{definition}
\newtheorem{defi}[trm]{Definition}

\def\OO{\mathcal{O}}

\def\cM{\mathcal{M}}
\def\cR{\mathcal{R}}
\def\rr{\overline{\mathcal{R}}}

\def\Pic0{{\rm Pic}^0(X)}

\begin{document}
\title{Irreducibility of a universal Prym-Brill-Noether locus}
\author{Andrei Bud}
\date{}
\maketitle
\begin{abstract}
	For genus $g = \frac{r(r+1)}{2}+1$, we prove that via the forgetful map, the universal Prym-Brill-Noether locus $\mathcal{V}^r_g$ has a unique irreducible component dominating the moduli space $\mathcal{R}_g$ of Prym curves. 
\end{abstract}
 
\section{Introduction} The moduli space $\cR_g$ of Prym curves was brought to the attention of algebraic geometers by Mumford in his influential paper \cite{MumfordPrym}, as a way of understanding principally polarized Abelian varieties. For an element $[C,\eta]$ of $\cR_g$ we let $\pi\colon \widetilde{C}\rightarrow C$ be the associated double cover and let $\mathrm{Nm}_\pi\colon \mathrm{Pic}^{2g-2}(\widetilde{C}) \rightarrow \mathrm{Pic}^{2g-2}(C)$ be the norm map of this morphism of curves. In this situation, the preimage of $\omega_C$ consists of two disjoint varieties 
\[ P^+ = \left\{L\in \mathrm{Pic}^{2g-2}(\widetilde{C}) \ | \ \mathrm{Nm}(L) = \omega_C \ \mathrm{and} \ h^0(\widetilde{C}, L) \equiv 0\ (\mathrm{mod} \ 2) \right\} \]
and 
\[ P^- = \left\{L\in \mathrm{Pic}^{2g-2}(\widetilde{C}) \ | \ \mathrm{Nm}(L) = \omega_C \ \mathrm{and} \ h^0(\widetilde{C}, L) \equiv 1\ (\mathrm{mod} \ 2) \right\}. \]
isomorphic to the Prym variety in $\mathrm{Pic}^0(\widetilde{C})$.

Following this development, Welters emphasized in \cite{Welters} that Prym-Brill-Noether theory can be employed in order to understand the geometry of subvarieties of Prym varieties. More precisely, he considered the loci 
\[V^r(C,\eta) \coloneqq \left\{ L \in \mathrm{Pic}^{2g-2}(\widetilde{C}) \ | \ \mathrm{Nm}(L)\cong \omega_C, \ h^0(\widetilde{C}, L) \geq r+1, \ \mathrm{and} \ h^0(\widetilde{C}, L) \equiv r+1\ (\mathrm{mod} \ 2)  \right\}\]
in order to study the singularities of the theta divisor of the associated Prym variety. The relation between Prym-Brill-Noether theory and the study of singularities of theta divisors piqued the interest of other mathematicians. 
The two papers \cite{Welters} and \cite{Bertram} showed that when $g \geq \frac{r(r+1)}{2}+1$, the locus $V^r(C,\eta)$ is non-empty of dimension at least $g-1-\frac{r(r+1)}{2}$. In addition, for a generic $[C,\eta]\in \cR_g$, the locus $V^r(C,\eta)$ has exactly this dimension when $g\geq\frac{r(r+1)}{2}+1$ and is empty when $g < \frac{r(r+1)}{2}+1$, see \cite{SchwarzPrym}. Subsequently in \cite{DeConciniPragacz}, De Concini and Pragacz viewed $V^r(C,\eta)$ as a Lagrangian degeneracy locus (cf. \cite{mumfordtheta}) and computed the class of $V^r(C,\eta)$ in the Prym variety when it has the expected dimension $g-1-\frac{r(r+1)}{2}$. 

In recent years, two new perspectives for the study of Prym-Brill-Noether theory emerged. On one hand, it has been studied from the point of view of tropical geometry, see \cite{tropicalPBN-4authors} and \cite{tropicalPBN-Len_Ulirsch}, thus providing another proof for the dimension estimate of $V^r(C,\eta)$ for a generic $[C,\eta]$ and, on the other hand, from the perspective of moduli theory, with a view to understanding the birational geometry of $\cR_g$ for small values of $g$. It is natural to ask when $g \geq \frac{r(r+1)}{2}+1$ whether the universal Prym-Brill-Noether locus 
\[ \mathcal{V}^r_g\coloneqq \left\{[C,\eta,L] \ | \ [C,\eta]\in \mathcal{R}_g \ \mathrm{and} \ L \in V^r(C,\eta)\right\}\]
has a unique irreducible component dominating the moduli space $\cR_g$. This is true for $g > \frac{r(r+1)}{2}+1$ because the fibre above a general $[C,\eta]\in \cR_g$ is irreducible, see \cite[Exemples 6.2]{DebarreLefschetz}. However, as pointed out in \cite{JensenPaynesurvey}, this was not known for $g = \frac{r(r+1)}{2}+1$. The present paper aims at showing that when $g = \frac{r(r+1)}{2}+1$, the moduli space $\mathcal{V}^r_g$ has a unique irreducible component dominating $\cR_g$. In the interest of proving this result, we will consider the compactification $\rr_g$ of the moduli space of Prym curves $\cR_g$, see \cite{Casa} and \cite{FarLud}. Ultimately, we degenerate to the boundary locus of $\rr_g$ and employ the theory of limit linear series, adapted to our situation. 

\textbf{Acknowledgements:} I would like to thank my advisor Gavril Farkas for suggesting this problem and for all his support and guidance along the way. I also want to thank the anonymous referee for the suggested improvements to this paper.  
\section{Prym linear series}

Our goal in this section is to provide a suitable definition of Prym linear series and then use it to prove our main result. We will start by recalling some definitions regarding limit linear series, while referring the reader to \cite{limitlinearbasic} for a throughout study.  

First, recall that a $g^r_d$ on a smooth curve $Y$ is defined to be a pair $(V, L)$ of a degree $d$ line bundle $L$ together with an $(r+1)$-dimensional vector subspace $V \subseteq H^0(Y,L)$. This definition extends naturally to curves of compact type. 

Let $Y$ be a genus $g$ curve of compact type. A crude limit $g^r_d$ consists of a $g^r_d$ $(V_i, L_i)$ for every irreducible component $Y_i$ of $Y$, further satisfying the following property: 

$\bullet$ Let $q$ be a node of $Y$ connecting two irreducible components $Y_j$ and $Y_k$ and let $0 \leq a_0<\cdots < a_r \leq d$ and $0 \leq b_0<\cdots < b_r \leq d$ the vanishing orders at $q$ of the sections in $V_j$ and $V_k$ respectively. Then for any $0\leq i \leq r$ we have 
\[ a_i + b_{r-i} \geq d. \]
If for any node $q$ all the inequalities are in fact equalities, the limit linear system is called refined. 
   
Lastly, we define the Brill-Noether number associated to a smooth pointed curve and a $g^r_d$ on it. Let $(Y,p_1,\ldots,p_n)$ be a smooth genus $g$ curve together with $n$ points on it, and $l = (V,L)$ be a $g^r_d$ on $Y$. Let $0\leq a_0^i <\cdots < a_r^i\leq d$ the vanishing orders at $p_i$ of the sections in $V$. The Brill-Noether number of the $g^r_d$ with respect to the points $p_i$ is defined as 
\[ \rho(l, p_1,\ldots,p_n) \coloneqq g-(r+1)(g-d+r) - \sum_{i=1}^{n}\sum_{j=0}^{r} a^i_j + n\cdot \frac{r(r+1)}{2}. \] 

Having these definitions, we are ready to particularize to our situation. 

Let $[C,\eta]\in \cR_g$ be a generic Prym curve. Then, we know from \cite[Lemma 3.2]{Welters} that a generic element $L\in V^r(C,\eta)$ satisfies $h^0(\widetilde{C},L) = r+1$. Moreover, when $g= \frac{r(r+1)}{2}+1$ we know from \cite[Theorem 1.1]{SchwarzPrym} that all $L\in V^r(C,\eta)$ satisfy $h^0(\widetilde{C},L) = r+1$. In particular, the line bundle $L$ can be viewed as a $g^r_{2g-2}$ on the curve $\widetilde{C}$. Furthermore, up to restricting to an open subset, we can view all irreducible components of $\mathcal{V}^r_g$ dominating $\cR_g$ as contained in the moduli space $\mathcal{G}^r_{2g-2}(\cR_g)$ parametrizing limit $g^r_{2g-2}$ over double covers $[\pi\colon\widetilde{C}\rightarrow C]$ where $\widetilde{C}$ is of compact type. We ask what points can appear in the compactification of $\mathcal{V}^r_{g}$ inside this space.  

Let $[\pi\colon \widetilde{C}\rightarrow C] \in \rr_g$ such that $C$ is of compact type and admits a unique irreducible component $X$ satisfying $\eta_X \ncong \OO_X$. For this component $X$, we denote by $p^X_1,\ldots, p^X_{s_X}$ its nodes and by $g^X_1,\ldots, g^X_{s_X}$ the genera of the connected components of $C\setminus X$ glued to $X$ at these points.  For an irreducible component $Y$ of $C$, different from $X$, we denote by $q^Y$ the node glueing $Y$ to the connected component of $C\setminus Y$ containing $X$, and by $p^Y_1,\ldots, p_{s_Y}^Y$ the other nodes of $Y$. We denote by $g^Y_0, g^Y_1,\ldots, g^Y_{s_Y}$ the genera of the connected components of $C\setminus Y$ glued to $Y$ at these points. 

Using the above notations, we can define the concept of a Prym limit $g^r_{2g-2}$: 
\begin{defi}
	 A Prym limit $g^r_{2g-2}$, denoted $L$, is a crude limit $g^r_{2g-2}$ on $\widetilde{C}$ satisfying the following two conditions: 
	\begin{enumerate}
		\item For the unique component $\widetilde{X}$ of $\widetilde{C}$ above $X$, the $\widetilde{X}$-aspect $L_{\widetilde{X}}$ of $L$ satisfies
		 \[Nm_{\pi_{|\widetilde{X}}} L_{\widetilde{X}} \cong \omega_X(\sum_{i=1}^s2g^X_ip^X_i).\]
		\item For a component $Y$ of $C$ different from $X$, we denote by $Y_1$ and $Y_2$ the two irreducible components of $\widetilde{C}$ above it. We identify these two components with $Y$ via the map $\pi$. With this identification the $Y_1$ and $Y_2$ aspects of $L$ satisfy: 
		\[ L_{Y_1}\otimes L_{Y_2} \cong \omega_Y\text{\large(}(2g-2+2g^Y_0)q^Y + \sum_{i=1}^sg_i^Yp_i^Y\text{\large)}. \]
	\end{enumerate} 
	
\end{defi}

Because the points in the boundary need to respect the norm condition, we immediately obtain that: 
\begin{lm}
	Let $[\pi\colon \widetilde{C}\rightarrow C]\in \rr_g$ with $\widetilde{C}$ of compact type and let $\overline{\mathcal{V}}^r_g$ the closure of $\mathcal{V}^r_g$ inside $\mathcal{G}^r_{2g-2}(\cR_g)$. Then the fibre of the map $\overline{\mathcal{V}}^r_g\rightarrow \rr_g$ over the point $[\pi\colon \widetilde{C}\rightarrow C]$ is contained in the locus of Prym limit $g^r_{2g-2}$ on $[\pi\colon \widetilde{C}\rightarrow C]$. 
\end{lm} 

We are now ready to use a degeneration argument in order to prove our main result. 

\begin{trm}
	When $g = \frac{r(r+1)}{2} +1$, the space $\mathcal{V}^r_g$ has a unique irreducible component dominating $\mathcal{R}_g$.  
\end{trm}
\begin{proof} We consider the boundary divisor $\Delta_1 \subseteq \overline{\mathcal{R}}_g$ whose generic point is of the form $[Y\cup_xE, \OO_Y, \eta_E \neq \OO_E]$, where $Y$ and $E$ are components of genus $g-1$ and $1$ respectively. 
	Let $[Y_1\cup_{x_1}\widetilde{E}\cup_{x_2}Y_2 \rightarrow Y\cup_xE]$ be the double cover associated to a generic element of $\Delta_1$. We want to describe the locus of Prym limit $g^r_{2g-2}$'s on such a double cover. 
	
	Let $L$ be a Prym limit $g^r_{2g-2}$ on $[Y_1\cup_{x_1}\widetilde{E}\cup_{x_2}Y_2 \rightarrow Y\cup_xE]$. The additivity of the Brill-Noether numbers implies: 
	\[ \rho(2g-1,r,2g-2) = -r \geq \rho(L_{Y_1},x_1) + \rho(L_{\widetilde{E}}, x_1,x_2) + \rho(L_{Y_2},x_2). \]
	But we know from \cite[Theorem 1.1]{EisenbudHarrisg>23} and \cite[Proposition 1.4.1]{FarkasThesis} that $ \rho(L_{Y_1},x_1) \geq 0, \ \rho(L_{Y_2},x_2) \geq 0$ and $\rho(L_{\widetilde{E}}, x_1,x_2) \geq -r$. It is clear that these are in fact equalities and $L$ is a refined limit $g^r_{2g-2}$. 
	
	We denote by $0\leq a_0 < a_1<\cdots <a_r \leq 2g-2$ and $0\leq b_0 < b_1<\cdots <b_r \leq 2g-2$ the vanishing orders for the $Y_1$ and $Y_2$ aspects respectively. The equality $\rho(L_{\widetilde{E}}, x_1,x_2) = -r$ implies that $a_i + b_{r-i} = 2g-2$ for all $0\leq i\leq r$. 
	
	The genericity of $[Y_2,x_2] \in \cM_{g-1,1}$ together with $ \rho(L_{Y_2},x_2) = 0$ imply that $h^0(Y_2, L_{Y_2}(-b_ix_2)) = r+1-i$ for all $0\leq i \leq r$. Using that $L_{Y_1}\otimes L_{Y_2} \cong \omega_Y(2g\cdot x)$ and the Riemann-Roch theorem we obtain 
	\[ h^0\text{\large(}Y_1, L_{Y_1}(-(2+a_{r-i})q)\text{\large)} = g+r-1-a_{r-i}-i.\]  
	Choosing $i=0$ we get $a_r = g+r-1$. Inverting the roles of the $a_i$'s and $b_i$'s we obtain that $a_0 = g-r-1$. Because we have the divisorial equivalences 
    \[ a_ix_1 +b_{r-i}x_2 \equiv  a_jx_1 +b_{r-j}x_2 \] 
    on the elliptic curve $E$ for every $0\leq i,j \leq r$, we obtain that $a_i-a_{i-1} \geq 2$ for every $1\leq i\leq r$. This implies that $a_i = g-r+2i-1$ for every $0\leq i \leq r$.
	
	We now view the moduli space $\cM_{g-1,1}$ as embedded in $\rr_g$ via the map $\pi\colon \cM_{g-1,1} \rightarrow \rr_g$ sending a pointed curve $[Y,x]\in \cM_{g-1,1}$ to $[Y\cup_xE, \OO_Y, \eta_E]$ where $[E,x]$ is a generic elliptic curve and $\eta_E$ is a 2-torsion line bundle on $E$. For the ramification sequence $\alpha = (g-r-1,\ldots, g-1)$ associated to the vanishing orders 
	$ a = (a_0,\ldots, a_r) = (g-r-1,\ldots, g+r-1)$, we consider the locus $\mathcal{G}^r_{2g-2}(\alpha)$ parametrizing pairs $[C,p,L]$ where $[C,p]\in\cM_{g-1,1}$ and $L$ is a $g^r_{2g-2}$ having vanishing orders greater or equal to $a$ at the point $p$.  Then the locus of Prym limit $g^r_{2g-2}$ over $\mathrm{Im}(\pi)$ is birationally isomorphic to $\mathcal{G}^r_{2g-2}(\alpha)$.
	
	We know from \cite[Lemma 3.6]{EisenbudHarrisBN-1} that $\mathcal{G}^r_{2g-2}(\alpha)$ has a unique irreducible component dominating $\cM_{g-1,1}$. Moreover 
	\[ \mathrm{deg}\text{\large (} \mathcal{G}^r_{2g-2}(\alpha) \rightarrow \cM_{g-1,1}\text{\large )}  = 2^{\frac{r(r-1)}{2}}\cdot(g-1)! \cdot \prod_{i=1}^r \frac{(i-1)!}{(2i-1)!} \]
	as stated on the second page of \cite{FarkasCastel}. 
	On the other hand we have from \cite[Theorem 9]{DeConciniPragacz} that 
	\[ \mathrm{deg}\text{\large (} \mathcal{V}^r_g \rightarrow \cR_g\text{\large )}  = 2^{\frac{r(r-1)}{2}}\cdot(g-1)! \cdot \prod_{i=1}^r \frac{(i-1)!}{(2i-1)!}. \]. 
	
	We conclude that all dominant irreducible components of $\mathcal{V}^r_g$ contain $\mathcal{G}^r_{2g-2}(\alpha)$ in their closure. From this we get that each such component map to $\cR_g$ with degree at least $2^{\frac{r(r-1)}{2}}\cdot(g-1)! \cdot \prod_{i=1}^r \frac{(i-1)!}{(2i-1)!}$, implying unicity.
\end{proof}

\bibliography{main}

\begin{thebibliography}{CLRW20}

\bibitem[BCF04]{Casa}
E.~Ballico, C.~Casagrande, and C.~Fontanari.
\newblock Moduli of {P}rym curves.
\newblock {\em Documenta Mathematica}, \textbf{9}:265--281, 2004.

\bibitem[Ber87]{Bertram}
A.~Bertram.
\newblock An existence theorem for {P}rym special divisors.
\newblock {\em Inventiones mathematicae}, \textbf{90}:669--671, 1987.

\bibitem[CLRW20]{tropicalPBN-4authors}
S.~Creech, Y.~Len, C.~Ritter, and D.~Wu.
\newblock {Prym–{B}rill–{N}oether {L}oci of {S}pecial {C}urves}.
\newblock {\em International Mathematics Research Notices},
  \textbf{2022}:2688--2728, 2020.

\bibitem[DCP95]{DeConciniPragacz}
C.~De~Concini and P.~Pragacz.
\newblock On the class of {B}rill-{N}oether loci for {P}rym varieties.
\newblock {\em Mathematische Annalen}, \textbf{302}:687--698, 1995.

\bibitem[Deb00]{DebarreLefschetz}
O.~Debarre.
\newblock Th\'eor\`emes de {L}efschetz pour les lieux de d\'eg\'en\'erescence.
\newblock {\em Bulletin de la Soci\'et\'e Math\'ematique de France},
  \textbf{128}:283--308, 2000.

\bibitem[EH86]{limitlinearbasic}
D.~Eisenbud and J.~Harris.
\newblock Limit linear series: {B}asic theory.
\newblock {\em Inventiones mathematicae}, \textbf{85}:337--372, 1986.

\bibitem[EH87]{EisenbudHarrisg>23}
D.~Eisenbud and J.~Harris.
\newblock The {K}odaira dimension of the moduli space of curves of genus $\geq$
  23.
\newblock {\em Inventiones mathematicae}, \textbf{90}:359--387, 1987.

\bibitem[EH89]{EisenbudHarrisBN-1}
D.~Eisenbud and J.~Harris.
\newblock Irreducibility of some families of linear series with
  {B}rill-{N}oether number {-1}.
\newblock {\em Annales scientifiques de l'\'Ecole Normale Sup\'erieure},
  \textbf{22}:33--53, 1989.

\bibitem[Far00]{FarkasThesis}
G.~Farkas.
\newblock The birational geometry of the moduli space of curves.
\newblock {\em Academisch Proefschrift, Universitet van Amsterdam}, 2000.

\bibitem[FL10]{FarLud}
G.~Farkas and K.~Ludwig.
\newblock The {K}odaira dimension of the moduli space of {P}rym varieties.
\newblock {\em Journal of the European Mathematical Society},
  \textbf{12}:755--795, 2010.

\bibitem[FT16]{FarkasCastel}
G.~Farkas and N.~Tarasca.
\newblock Pointed {C}astelnuovo numbers.
\newblock {\em Mathematical Research Letters}, \textbf{23}:389--404, 2016.

\bibitem[JP21]{JensenPaynesurvey}
D.~Jensen and S.~Payne.
\newblock Recent {D}evelopments in {B}rill-{N}oether {T}heory.
\newblock {\em Preprint, arXiv:2111.00351}, 2021.

\bibitem[LU21]{tropicalPBN-Len_Ulirsch}
Y.~Len and M.~Ulirsch.
\newblock {Skeletons of {P}rym varieties and {B}rill–{N}oether theory}.
\newblock {\em Algebra {$\&$} Number Theory}, \textbf{15}:785--820, 2021.

\bibitem[Mum71]{mumfordtheta}
D.~Mumford.
\newblock Theta characteristics of an algebraic curve.
\newblock {\em Annales Scientifiques de l'\'{E}cole Normale Sup\'{e}rieure},
  \textbf{4}:181--192, 1971.

\bibitem[Mum74]{MumfordPrym}
D.~Mumford.
\newblock Prym varieties {I}.
\newblock {\em Contributions to analysis}, pages 325--350, 1974.

\bibitem[Sch17]{SchwarzPrym}
Irene Schwarz.
\newblock Brill–{N}oether theory for cyclic covers.
\newblock {\em Journal of Pure and Applied Algebra}, \textbf{221}:2420--2430,
  2017.

\bibitem[Wel85]{Welters}
G.~Welters.
\newblock A theorem of {G}ieseker-{P}etri type for {P}rym varieties.
\newblock {\em Annales Scientifiques de l'École Normale Supérieure},
  \textbf{18}:671--683, 1985.

\end{thebibliography}
\bibliographystyle{alpha}
\Addresses
\end{document}